\documentclass[%
 aip,
 amsmath,amssymb,
 reprint,%
]{revtex4-1}

\usepackage{graphicx}
\usepackage{dcolumn}
\usepackage{bm}
\usepackage{amsmath,amsthm} 
\usepackage{overpic}
\usepackage{tikz} 
\usepackage{enumerate}

\usepackage[utf8]{inputenc}
\usepackage[T1]{fontenc}
\usepackage{mathptmx}
\usepackage{etoolbox}

\usepackage{xcolor}
\newcommand{\tc}[2]{\textcolor{#1}{#2}}

\definecolor{red}{rgb}{.5,0,0} \newcommand{\red}[1]{\tc{red}{#1}}
\definecolor{green}{rgb}{0,.5,0} 
\definecolor{blue}{rgb}{0,0,0.5}

\newcommand{\ER}{Erd\H{o}s--R\'{e}nyi }
\newcommand{\Erdos}{Erd\H{o}s--R\'{e}nyi }

\makeatletter
\def\@email#1#2{%
 \endgroup
 \patchcmd{\titleblock@produce}
  {\frontmatter@RRAPformat}
  {\frontmatter@RRAPformat{\produce@RRAP{*#1\href{mailto:#2}{#2}}}\frontmatter@RRAPformat}
  {}{}
}%
\makeatother

\newtheorem{thm}{Theorem}
\newtheorem{theorem}[thm]{Theorem}
\newtheorem{lem}[thm]{Lemma}
\newtheorem{lemma}[thm]{Lemma}

\newtheorem{cor}[thm]{Corollary}

\begin{document}

\preprint{AIP/123-QED}

\title[Synchronization of random networks]{A global synchronization theorem for oscillators on a random graph}

\author{Martin Kassabov}
\altaffiliation{Department of Mathematics, Cornell University, Ithaca, NY 14853}%
\author{Steven H. Strogatz}
\altaffiliation{Department of Mathematics, Cornell University, Ithaca, NY 14853}%
\author{Alex Townsend}
\altaffiliation{Department of Mathematics, Cornell University, Ithaca, NY 14853}%

\date{\today}

\begin{abstract}
Consider $n$ identical Kuramoto oscillators on a random graph. Specifically, consider \ER random graphs in which any two oscillators are bidirectionally coupled with unit strength, independently and at random, with probability $0\leq p\leq 1$. We say that a network is globally synchronizing if the oscillators converge to the all-in-phase synchronous state for almost all initial conditions.  Is there a critical threshold for $p$ above which global synchrony is extremely likely but below which it is extremely rare? It is suspected that a critical threshold exists and is close to the so-called connectivity threshold, namely, $p\sim \log(n)/n$ for $n \gg 1$. Ling, Xu, and Bandeira made the first progress toward proving a result in this direction: they showed that if $p\gg \log(n)/n^{1/3}$, then \ER networks of Kuramoto oscillators are globally synchronizing with high probability as $n\rightarrow\infty$.
Here we improve that result by showing that $p\gg \log^2(n)/n$ suffices.  Our estimates are explicit: for example, we can say that there is more than a $99.9996\%$ chance that a random network with $n = 10^6$ and $p>0.01117$ is globally synchronizing. 
\end{abstract}

\maketitle

\begin{quotation}
Random graphs are fascinating topologies on which to study the dynamics of coupled oscillators. Despite the statistical nature of random graphs, coherent synchronization seems to be ubiquitous above a critical threshold.  To investigate this, we consider the homogeneous version of the Kuramoto model in which all $n$ oscillators have the same intrinsic frequency.  For simplicity, the oscillators are arranged on an \ER random graph in which any two oscillators are coupled with unit strength by an undirected edge, independently at random, with some probability $0\leq p \leq 1$; however, our proof strategy can be applied to other random graph models too. We say that a network is globally synchronizing if the oscillators converge to the all-in-phase synchronous state for almost all initial conditions. For what values of $p$ is an \ER random network very likely to be globally synchronizing? Here we prove that $p\gg \log^2(n)/n$ as $n\rightarrow\infty$ is a sufficient condition. Our proof uses trigonometric inequalities and an amplification argument involving the first two moments of the oscillator phase distribution that must hold for any stable phase-locked state.  Specifically, we show that the spectral norms of the mean-centered adjacency and graph Laplacian matrix can be used to guarantee that a network is globally synchronizing. Our analysis is explicit, and we can reason about random networks of finite, practical size. 
\end{quotation}

\section{\label{sec:level1}Introduction}
Networks of coupled oscillators have long been studied in biology, physics, engineering, and nonlinear dynamics.~\cite{winfree1967biological,winfree1980geometry, kuramoto1984chemical, mirollo1990synchronization, watanabe1994constants, pikovsky2003synchronization,strogatz2003sync, dorfler2014synchronization, pikovsky2015dynamics, rodrigues2016kuramoto, abrams2016introduction, boccaletti2018synchronization, matheny2019exotic} 
Recently they have begun to attract the attention of other communities as well. For example, oscillator networks have been recognized as having the potential to yield ``beyond-Moore's law'' computational devices~\cite{CsabaPorod} for graph coloring,~\cite{Wu} image segmentation,~\cite{Novikov} and approximate maximum graph cuts.~\cite{Steinerberger} They have also become a model problem for understanding the global convergence of gradient descent in nonlinear optimization.~\cite{ling2018landscape} 

In such settings, global issues come to the fore. When performing gradient descent, for instance, one typically wants to avoid getting stuck in local minima. Conditions to enforce convergence to the global minimum then become desirable. Likewise, if an oscillator network has only a single, globally attracting state, we know exactly how the system will behave in the long run. 

We say that a network of oscillators \emph{globally synchronizes} if it converges to a state for which all the oscillators are in phase, starting from all initial conditions except a set of measure zero. 
Until recently, only a few global synchronization results were known for networks of  oscillators.~\cite{mirollo1990synchronization, watanabe1994constants} These results were restricted to complete graphs, in which each oscillator is coupled to all the others. In the past decade, however, several  advances have been made for a  wider class of network structures, starting with work by  Taylor,~\cite{taylor2012there} who proved that if each oscillator in a network of identical Kuramoto oscillators is connected to at least 94 percent of the others, the network will fall into perfect synchrony for almost all initial conditions, no matter what the topology of the network is like in other respects. Taylor's result was strengthened by Ling, Xu, and Bandeira~\cite{ling2018landscape} in 2018 and further progress has been made since then.~\cite{lu2020synchronization, kassabovtownsendstrogatz}

Ling {\em et al.}~\cite{ling2018landscape} also made a seminal advance in the study of \emph{random} networks. They considered identical Kuramoto oscillators on an \ER random graph,~\cite{ErdosRenyi} in which any two oscillators are coupled with unit strength, independently and at random, with probability $0\leq p\leq 1$; otherwise they are uncoupled. They showed that if $p\gg \log(n)/n^{1/3}$, then with high probability the network is globally synchronizing as $n\rightarrow\infty$.

The open question is to find and prove the sharpest result along these lines. Intuitively, as one increases the value of $p$ from $0$ to $1$, one expects to find a critical threshold above which global synchrony is extremely likely and below which 
it is extremely rare. At the very least, for global synchrony to be ubiquitous, $p$ must be large enough to ensure that the random network is connected, and from the theory of random graphs~\cite{ErdosRenyi} we know that connectedness occurs with high probability once $p>(1+\varepsilon)\log(n)/n$ for any $\varepsilon>0$.  So the critical threshold for global synchronization cannot be any smaller than this connectivity threshold, and is apt to be a little above. On the basis of numerical evidence, Ling {\em et al.}~\cite{ling2018landscape} conjectured, and we agree with them, that if $p\gg \log(n)/n$ then \ER graphs are globally synchronizing as $n\rightarrow\infty$, but nobody has proven that yet. The challenge is to see how close one can get. In this paper, we come within a factor of $\log(n)$ and prove that $p\gg \log^2(n)/n$ is sufficient to give global synchrony with high probability. With the aid of a computer, we have convincing evidence that $p\gg\log(n)/n$ is sufficient as $n\rightarrow \infty$ (see Section~\ref{sec:compOpt}). 

We study the homogeneous Kuramoto model~\cite{jadbabaie2004stability, wiley2006size, taylor2012there} in which each oscillator has the same frequency $\omega$. By going into a rotating frame at this frequency, we can set $\omega = 0$ without loss of generality. Then phase-locked states in the original frame correspond to equilibrium states in the rotating frame. So, to explore the question that concerns us, it suffices to study the following simplified system of identical Kuramoto oscillators:
\begin{equation} 
\frac{d\theta_j}{dt} = \sum_{k=1}^{n} A_{jk} \sin\!\left(\theta_k - \theta_j\right), \quad 1\leq j\leq n, 
\label{eq:dynamical}
\end{equation} 
where $\theta_j(t)$ is the phase of oscillator $j$ (in the rotating frame) and the adjacency matrix $A$ is randomly generated. In particular, $A_{jk} = A_{kj} = 1$ with probability $p$, with $A_{jk} = A_{kj} = 0$ otherwise, independently for $1\leq j, k\leq n$. Thus, all interactions are assumed to be symmetric, equally attractive, and of unit strength. We take the unusual convention that the network can have self-loops so that oscillator $i$ is connected to itself with probability $p$, i.e., $\mathbb{P}\left[A_{jj} = 1\right] = p$. Since $\sin(0)=0$, this convention does not alter the dynamics of the oscillators but it does make proof details easier to write down. 

Finally, since the adjacency matrix $A$ is symmetric, we know that~\eqref{eq:dynamical} is a gradient system.~\cite{wiley2006size,jadbabaie2004stability} Thus, all the attractors of~\eqref{eq:dynamical} are equilibrium points, which means we do not need to concern ourselves with the possibility of more complicated attracting invariant sets like limit cycles, tori, or strange attractors. 

We find it helpful to visualize the oscillators on the unit circle, where oscillator $j$ is positioned at the coordinate $(\cos(\theta_j),\sin(\theta_j))$. From this perspective, a network is globally synchronizing if starting from any initial positions (except a set of measure zero), all the oscillators eventually end up at the same point on the circle. Due to periodicity, we assume that the phases, $\theta_j$, take values in the interval $[-\pi,\pi)$. 

One cannot determine whether a network is globally synchronizing by numerical simulation of~\eqref{eq:dynamical}, as it is impossible to try all initial conditions. Of course, one can try millions of random initial conditions of the oscillators' phases and then watch the dynamics of~\eqref{eq:dynamical}. But even if all observed initial states eventually fall into the all-in-phase state, one cannot conclude the system is globally  synchronizing because other stable equilibria could still exist; their basins might be minuscule but could nevertheless have positive measure. 

With that caveat in mind, we note that such numerical experiments have been conducted, and they suggest that $p \gg \log(n)/n$ is sufficient for global synchronization.~\cite{ling2018landscape} In this paper, we investigate global synchrony via theoretical study. We show that $p\gg \log^2(n)/n$ is good enough to ensure global synchronization with high probability, improving on $p\gg \log(n)/n^{1/3}$ proved by Ling {\em et al.},~\cite{ling2018landscape} and bringing us closer to the connectivity threshold of \ER graphs.

Although we focus on \ER graphs, many of the inequalities we derive hold for any random or deterministic network. To highlight this, we state many of our findings for a general graph $G$ and a general parameter $p\in\mathbb{R}$. In the end, we restrict ourselves back to \ER graphs and take $p$ to be the probability of a connection between any two oscillators. 

Our results depend on both the adjacency matrix $A$ and the graph Laplacian matrix $L = D - A$, where $D$ is a diagonal matrix and $D_{ii}$ is the degree of vertex $i$ (counting self-loops). For any $p\in\mathbb{R}$, denote the shifted adjacency and graph Laplacian matrix by 
\begin{equation} 
\Delta_A = A - pJ_n,\qquad \Delta_L = L - pJ_n + npI_n,
\label{eq:DeltaNotation} 
\end{equation} 
where $J_n$ is the $n\times n$ matrix of all ones and $I_n$ is the $n\times n$ identity matrix. It is worth noting that for \ER graphs, the shifts are precisely the expectation of the matrices as $\mathbb{E}\left[A\right] =  pJ_n$ and $\mathbb{E}\left[L\right] = pJ_n - npI_n$. Remarkably, we show that the global synchrony of a network can be guaranteed by ensuring that the spectral norms $\|\Delta_A\|$ and $\|\Delta_L\|$ satisfy particular inequalities. The spectral norm of a symmetric matrix is the maximum eigenvalue in absolute magnitude, and $\|\Delta_A\|$ and $\|\Delta_L\|$ are extensively studied in the random matrix literature.~\cite{FK,Vu} What is remarkable is that no other information about the network's structure is needed; the norms of these two matrices alone encapsulate the structural aspects that matter for global synchronization. We also find it appealing that the spectral norm of the graph Laplacian matrix appears naturally in our analysis, as it has been used previously to study the dynamics of networks of oscillators~\cite{Pecora} as well as diffusion on graphs.

While we focus in this paper on global synchronization of identical oscillators, Medvedev and his collaborators have considered other aspects of synchronization for {\em non}-identical oscillators on \ER graphs, as well as on other graphs such as Cayley and Watts--Strogatz graphs, often in the continuum limit.~\cite{Medvedev1,Medvedev2,Medvedev3, ChibaMedvedevMizuhara} Their findings have a similar overarching message that synchronization occurs spontaneously above a critical threshold. 

\section{Overview of our proof strategy}
To prove that a random graph is globally synchronizing with high probability, we bridge the gap between spectral graph theory and coupled oscillator theory. The literature contains many good probabilistic estimates for the spectral norm of a random graph's adjacency and graph Laplacian matrices, which we use to control the long-time dynamics of~\eqref{eq:dynamical}. The key to our proof is to establish  conditions on these two spectral norms that force any stable equilibrium to have phases that lie within a half-circle. Confining the phases in this way then guarantees global synchronization, because it is known that the only stable equilibrium of~\eqref{eq:dynamical} with phases confined to a half-circle is the all-in-phase synchronized state.~\cite{jadbabaie2004stability, dorfler2014synchronization}  

Our first attempt at controlling the phases is the inequality we state below as~\eqref{eq:FirstBound}, which can be used to guarantee that very dense \ER graphs are globally synchronizing with high probability; as an aside, when $p=1$ this inequality provides a new proof that a complete graph of identical Kuramoto oscillators is globally synchronizing.~\cite{watanabe1994constants} A similar inequality to~\eqref{eq:FirstBound} was derived by Ling, Xu, and Bandeira~\cite{ling2018landscape} to show that $p\gg \log(n)/n^{1/3}$ is sufficient for global synchrony with high probability.

To improve on~\eqref{eq:FirstBound}, our argument becomes more intricate. We carefully examine the possible distribution of edges between oscillators whose phases lie on different arcs of the circle, and show that any equilibrium is destabilized if there are too many edges between oscillators that have disparate phases (see Lemma~\ref{lem:Edges} and Figure~\ref{fig:CsetEdges}).

\section{Bounds on the order parameter and its higher-order moments}\label{sec:OrderParameter}
An important quantity in the study of Kuramoto oscillators is the so-called complex order parameter, $\rho_1$. The magnitude of $\rho_1$ is between $0$ and $1$ and measures the synchrony of the oscillators in the network. We find it useful to also look at second-order moments of the oscillator distribution for analyzing the synchrony of random networks. Higher-order moments are also called Daido order parameters and can be used to analyze oscillators with all-to-all coupling, corresponding to a complete graph.~\cite{daido1992order,Perez_97,pikovsky2015dynamics,terada2020linear}

For an equilibrium $\theta = (\theta_1,\ldots,\theta_n)$, we define the first- and second-order moments as  
\[
\rho_{1} = \frac{1}{n}\sum_{j} e^{i\theta_j}, \qquad \rho_{2} = \frac{1}{n}\sum_{j} e^{2i\theta_j}.
\]
(For convenience, we use the notation $\sum_{j}$ to mean $\sum_{j=1}^n$.) Without loss of generality, we may assume that the complex order parameter $\rho_{1}$ is real-valued and non-negative. To see this, write $\rho_{1} = |\rho_{1}|e^{i\psi}$ for some $\psi$. Then, $\hat{\theta} = (\theta_1-\psi,\ldots,\theta_n-\psi)$ is also an equilibrium of~\eqref{eq:dynamical} with the same stability properties as $\theta$ since~\eqref{eq:dynamical} is invariant under a global shift of all phases by $\psi$. Therefore, for the rest of this paper, we assume that $\psi=0$ for any equilibrium of interest with $0\leq \rho_1\leq 1$. When $\rho_1 = 1$, the oscillators are in pure synchrony and when $\rho_1\approx 0$, the phases are scattered around the unit circle without a dominant phase.

To avoid working with complex numbers, it is convenient to consider the quantity $|\rho_2|^2$. For $m = 1, 2$, we have
\begin{equation} 
|\rho_m|^2 = \frac{1}{n^2}\left(\sum_k e^{im\theta_k}\sum_j e^{-im \theta_j}\right) = \frac{1}{n^2}\sum_{j,k} \cos(m(\theta_k-\theta_j)). 
\label{eq:SquaredMoments} 
\end{equation} 
By analyzing $\rho_1^2$ and $|\rho_2|^2$, one hopes to witness the rough statistics of an equilibrium to understand its potential for synchrony without concern for its precise pattern of phases. 

Let $q_\theta = (e^{i\theta_1},\dots, e^{i\theta_n})^\top$ and note that 
\[
\sum_{j,k}A_{jk} \cos(\theta_k -\theta_j) = \overline{q_\theta}^\top A q_\theta,
\]
where $A$ is the adjacency matrix, $\overline{q_\theta}$ is the complex conjugate of $q_\theta$, and ${}^\top$ denotes the transpose of a vector. Since $\cos^2(\theta_k -\theta_j) = \frac{1}{2} (\cos (2(\theta_k -\theta_j))+1)$, we have
\[
\sum_{j,k}A_{jk} \cos^2(\theta_k -\theta_j) = \frac{1}{2}\overline{q_{2\theta}}^\top A q_{2\theta} + \frac{1}{2}\mathbf{1}^\top A \mathbf{1},
\]
where $\mathbf{1}$ is the vector of all ones and $q_{2\theta} = (e^{i2\theta_1},\dots, e^{i2\theta_n})^\top$.

We would like to know when a stable equilibrium is close to the all-in-phase state, and we know this when $\rho_1$ is close to $1$.  Therefore, we begin by deriving a lower bound on $\rho_1$ for any stable equilibrium of~\eqref{eq:dynamical}. Similar inequalities involving $\rho_1^2$ and $|\rho_2|^2$ have been used to demonstrate that  sufficiently dense Kuramoto networks are globally synchronizing.~\cite{ling2018landscape,kassabovtownsendstrogatz}

\begin{lem} 
\label{lem:OrderParameters}
Let $G$ be a connected graph and $\theta$ a stable equilibrium of~\eqref{eq:dynamical}. For any $p>0$, we have
\begin{equation} 
\rho_1^2  \geq \frac{1 + |\rho_2|^2}{2} - \frac{2\|\Delta_A\|}{np},
\label{eq:LowerBoundRho1}
\end{equation} 
where the mean-shifted adjacency matrix, $\Delta_A$, is defined in~\eqref{eq:DeltaNotation}. 
\end{lem} 
\begin{proof} 
Since $\Delta_A = A - pJ_n$, we have 
$$
\overline{q_\theta}^\top A q_\theta = 
p\overline{q_\theta}^\top J_n q_\theta  + \overline{q_\theta}^\top \Delta_A  q_\theta.
$$
One finds that $\overline{q_\theta}^\top J_n q_\theta = n^2\rho_1^2$ by~\eqref{eq:SquaredMoments} and that $|\overline{q_\theta}^\top \Delta_A  q_\theta|\leq \|\Delta_A\| \|q_\theta\|^2\leq n\|\Delta_A\|$ so 
\begin{equation} 
\left| \sum_{j,k}A_{jk} \cos(\theta_k -\theta_j) - n^2p  \rho_1^2 \right| \leq n\|\Delta_A\|. 
\label{eq:DeviationOfTraceAX}
\end{equation} 
By the same reasoning for $\overline{q_{2\theta}}^\top A q_{2\theta}$, we find that
\[
\left| \sum_{j,k}A_{jk} \cos(2(\theta_k -\theta_j)) - n^2 p |\rho_2|^2 \right| \leq n\|\Delta_A\|.
\]
Moreover, $\mathbf{1}^\top A\mathbf{1} =  \mathbf{1}^\top \Delta_A\mathbf{1} + n^2p$, so $\left|\mathbf{1}^\top A\mathbf{1} - n^2p\right|\leq n\|\Delta_A\|$. We conclude that 
\begin{equation} 
\left| \sum_{j,k}A_{jk} \cos^2(\theta_k -\theta_j)- n^2p \frac{|\rho_2|^2+1}{2} \right| \leq n\|\Delta_A\|. 
\label{eq:DeviationOfTraceAY}
\end{equation} 
Since $\theta$ is stable equilibrium for~\eqref{eq:dynamical}, it is known that
\[
\sum_{j,k} A_{jk} \cos (\theta_k - \theta_j)(1 -\cos(\theta_k - \theta_j)) \geq 0
\] 
as shown by Ling {\em et al.}~\cite{ling2018landscape} on p.~1893 of their paper. From~\eqref{eq:DeviationOfTraceAX} and~\eqref{eq:DeviationOfTraceAY}, we must have
$$
n^2p\rho_1^2 + n\|\Delta_A\| \geq n^2p \frac{|\rho_2|^2+1}{2}-n\|\Delta_A\|, 
$$
which is equivalent to the statement of the lemma. 
\end{proof} 

To maximize the lower bound on $\rho_1^2$ from Lemma~\ref{lem:OrderParameters}, one can optimize over $p$. For a random graph where each edge has a fixed probability of being present, independently of the other edges, we usually just select $p$ to be that probability. Regardless, to make the lower bound on $\rho_1^2$ in Lemma~\ref{lem:OrderParameters} useful, we need to find a nontrivial lower bound on $|\rho_2|^2$, since this quantity appears in the right hand side of~\eqref{eq:LowerBoundRho1}. To obtain such a bound we use the following technical lemma. 

\begin{lem} 
Let $G$ be a connected graph and $\theta$ a stable equilibrium. We have 
\[
\|\Delta_A q_\theta \|^2  \geq  n^2p^2 \rho_1^2 \left( \sum_j \sin^2( \theta_j )+ \sum_{j, \cos (\theta_j) \leq 0} \cos^2 (\theta_j)\right),
\]
where $\|\cdot\|$ is the Euclidean norm. 
\label{lem:LowerBoundSin}
\end{lem} 
\begin{proof} 
Select any $j$ such that $1 \leq j \leq n$. From the fact that $\theta$ is an equilibrium, we have $\sum_k A_{jk} \sin(\theta_k -\theta_j) = 0$. Moreover, because the equilibrium is stable, we also have $\sum_k A_{jk} \cos(\theta_k-\theta_j) \geq 0$, which follows as the diagonal entries of the Hessian matrix must be nonnegative at a stable equilibrium (see (2.3) of Ling {\em et al.}~\cite{ling2018landscape}). These inequalities can be written as
$$
{\rm Re\,} (e^{-i\theta_j}e_j^\top A q_\theta) \geq 0
\quad \mbox{and }\quad
{\rm Im\,} (e^{-i\theta_j}e_j^\top A q_\theta) = 0,
$$
where $e_j$ is the $j$th unit vector. Since $\Delta_A = A - pJ_n$ and using that $J_n q_\theta = n \rho_1  \mathbf{1}$, we find that
\[
| {\rm Im\,}(e^{-i\theta_j}e_j^\top \Delta_A q_\theta)| =
np\rho_1 \left|\sin(\theta_j)\right|.
\]
If $\cos \theta_j \leq 0$, then we also have
\[
\begin{aligned} 
{\rm Re \,}(e^{-i\theta_j}e_j^\top& \Delta_A q_\theta ) =  \sum_k (A_{jk}-p) \cos(\theta_k-\theta_j) \\
=& \sum_k A_{jk} \cos(\theta_k-\theta_j) - np \rho_1 \cos \theta_j
\geq np \rho_1 \left|\cos \theta_j\right|.
\end{aligned} 
\]
The inequality in the lemma follows by squaring the above inequalities, summing over $j$, and noting that $|e^{-i\theta_j}| = 1$.
\end{proof} 

To see how Lemma~\ref{lem:LowerBoundSin} can be used to derive a lower bound on $|\rho_2|^2$ for any stable equilibrium state, we start by dropping the second sum in Lemma~\ref{lem:LowerBoundSin}.  By using the upper bound $\|\Delta_A q_\theta \|^2\leq n\|\Delta_A\|^2$, we find that $\sum_j \sin^2 (\theta_j)\leq \|\Delta_A\|^2/(np^2\rho_1^2)$.
Since $n{\rm Re}\!\left(\rho_{2}\right) = \sum_j \cos(2\theta_j) = \sum_j \left(1 - 2\sin^2(\theta_j)\right)$, we have the following lower bound on $|\rho_2|$:
\begin{equation} 
|\rho_2|\geq {\rm Re}(\rho_2) = \frac{1}{n}\sum_j \left(1 - 2\sin^2(\theta_j)\right) \geq 1 - \frac{2 \|\Delta_A\|^2  }{n^2p^2\rho_1^2}.
\label{eq:LowerBoundRho2}
\end{equation} 
From~\eqref{eq:LowerBoundRho1} and~\eqref{eq:LowerBoundRho2}, we observe that when $\|\Delta_A\| \ll np$, then $\rho_1$ and $|\rho_2|$ must both be close to $1$. Intuitively, this should mean that the corresponding stable equilibrium, $\theta$, must be close to the all-in-phase state with the possible exception of a small number of stray oscillators. However, our goal is to prove global synchrony, which is a more stringent condition, and we must completely rule out the existence of stray oscillators. 

To precisely control the number of stray oscillators, we define a set of indices for oscillators whose phases lie outside of a sector of half-angle $\phi$ (centered about the all-in-phase state):   
\begin{equation} 
C_\phi = \left\{k : \cos(\theta_k) \leq \cos(\phi), \,\, 1\leq k\leq n \right\}
\label{eq:ImportantSet} 
\end{equation} 
for any angle $0\leq \phi \leq \pi$ (see Figure~\ref{fig:Cset}).  If we can prove that $C_{\pi/2}$ is empty, then we know that all the phases in the equilibrium state lie strictly inside a half-circle. That would give us what we want, because by a basic theorem for the homogeneous Kuramoto model, the only stable equilibrium $\theta$ with all its phases confined to a half-circle is the all-in-phase state.~\cite{jadbabaie2004stability, dorfler2014synchronization} In terms of bounds, since $|C_{\pi/2}|$ is integer-valued, if we can show that $|C_{\pi/2}|<1$ then we know that $C_{\pi/2}$ is the empty set.
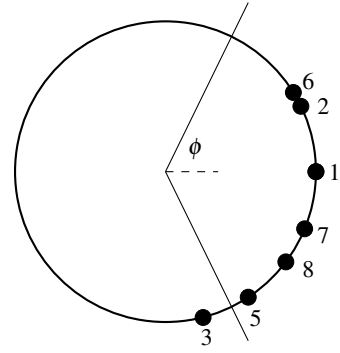
\begin{figure} 
\centering 
\begin{tikzpicture} 
\draw[thick,black] (0,0) circle [radius=2cm];
\filldraw[black] (2,0) circle (3pt);
\filldraw[black] (1.7,1.054) circle (3pt);
\filldraw[black] (1.8,0.87) circle (3pt);
\filldraw[black] (1.6,-1.2) circle (3pt);
\filldraw[black] (1.85,-0.76) circle (3pt);
\filldraw[black] (1.1,-1.67) circle (3pt);
\filldraw[black] (0.5,-1.94) circle (3pt);
\filldraw[black] (2,0) circle (3pt);
\draw[black] (0,0) -- (1.1,2.25);
\draw[black] (0,0) -- (1.1,-2.25);
\draw[black,dashed] (0,0) -- (.7,0);
\node at (2.25,0) {$1$};
\node at (2.1,0.87) {$2$};
\node at (0.55,-2.2) {$3$};
\node at (1.2,-1.95) {$5$};
\node at (1.9,1.2) {$6$};
\node at (2.1,-0.85) {$7$};
\node at (1.9,-1.3) {$8$};
\node at (.4,.3) {$\phi$};
\end{tikzpicture} 
\caption{When viewing the phases of the oscillators on the unit circle, the set $C_\phi$ contains the indices of stray oscillators whose phases have cosines less than or equal to $\cos(\phi)$. In the example shown here, we have $8$ oscillators, and only the oscillator with phase $\theta_3$ has strayed outside the sector of half-angle $\phi$. Hence $C_{\phi} = \left\{3\right\}$.}
\label{fig:Cset}
\end{figure} 

From Lemma~\ref{lem:LowerBoundSin} and $\|\Delta_A q_\theta \|^2\leq n\|\Delta_A\|^2$, we find that 
\begin{equation} 
n\|\Delta_A\|^2\geq n^2p^2\rho_1^2\sum_{j\in C_\phi}\sin^2(\theta_j)\geq n^2p^2\rho_1^2|C_\phi| \sin^2(\phi). 
\label{eq:FirstBound}
\end{equation} 
Therefore, by plugging in $\phi = \pi/2$, we see that $|C_{\pi/2}| \leq \|\Delta_A\|^2/(np^2\rho_1^2)$. Thus, if $\|\Delta_A\|^2 < np^2\rho_1^2$ then $C_{\pi/2}$ must be the empty set and the network is globally synchronizing. 

Unfortunately, this kind of reasoning is not sufficient to prove global synchrony for graphs of interest to us here. For example, for an \ER random graph, we know that $\|\Delta_A\|^2\approx 4p(1-p)n$ with high probability for large $n$ (see Section~\ref{sec:Erdos}). So the upper bound on $|C_{\pi/2}|$ is approximately $4(1-p)/(p\rho_1^2)$, which for $p<1/2$ is certainly not good enough to conclude that $C_{\pi/2}$ is empty.  Instead, we must further improve our bounds on $|C_{\pi/2}|$ by using a recursive refinement strategy that we refer to as an ``amplification" argument (see Section~\ref{sec:amplification}). 

\section{Bounds on the number of edges and sizes of sets in graphs}
The precise amplification argument that we use requires bounds on the number of edges and sizes of vertex sets of a graph expressed in terms of the spectral norms of $\Delta_A$ and $\Delta_L$. It is worth noting that these bounds hold for both deterministic and random graphs. For a vertex set $C$ of a graph $G$, we denote the characteristic vector by $v_C$, i.e., $(v_C)_j = 1$ if $j\in C$ and $(v_C)_j = 0$ if $j\not\in C$. We use $v_C^\top$ to denote the vector transpose of $v_C$. We write $|C|$ to be the number of vertices in $C$ and denote the number of edges in $G$ between two vertex sets $C$ and $C'$ as $E_{C,C'}$. For \ER graphs, one expects to have $E_{C,C'}\approx p |C| |C'|$. However, for our argument to work, expectations are not adequate; instead we need to have bounds on the difference between $E_{C,C'}$ and $p |C| |C'|$. The results in this section are proved by classical techniques and closely related bounds are well known.  We give the proofs of the precise inequalities that we need so that the paper is self-contained. 

We first show that for any vertex set $C$, the number of edges connecting a vertex in $C$ to another vertex in $C$ deviates from $p|C|^2$ by at most $\|\Delta_A\||C|$, for any $0\leq p\leq 1$.

\begin{lem}
Let $G$ be a graph of size $n$ with vertex set $V_G$ and adjacency matrix $A$. For any $0\leq p\leq 1$, we have
$$
\max_{C\subseteq V_G} \frac{\left| E_{C,C} - p |C|^2 \right|}{|C|} \leq \|\Delta_A\|,
$$
where $\Delta_A$ is defined in~\eqref{eq:DeltaNotation}.  
\label{lem:OnePartitioning}
\end{lem}
\begin{proof}
Let $v_C$ be the characteristic vector for $C$. By the min-max theorem,~\cite{Golub} we have 
$\max_{C\subseteq V_G} (v_C^\top\Delta_A v_C)/(v_C^\top v_C) \leq \|\Delta_A\|$.
Finally, we note that $\Delta_A = A - pJ_n$, $v_C^\top v_C = |C|$, $v_C^\top J_nv_C = |C|^2$, and $v_C^\top A v_C = E_{C,C}$.
\end{proof} 

Lemma~\ref{lem:OnePartitioning} controls the number of edges connecting a set of vertices. The following result controls the number of edges leaving a vertex set. We denote the vertices of $G$ that are not in $C_1$ by $V_G\setminus C_1$. 
\begin{lem}
Let $G$ be a graph of size $n$ with vertex set $V_G$ and graph Laplacian $L$. For any $0\leq p\leq 1$, we have
$$
\max_{C_1\subseteq V_G,C_2 = V_G\setminus C_1} \frac{\left| E_{C_1,C_2} - p |C_1||C_2| \rule{0pt}{2ex}\right|}{|C_1| |C_2|} \leq \frac{\|\Delta_L\|}{n}, 
$$
where $\Delta_L$ is defined in~\eqref{eq:DeltaNotation}. 
\label{lem:TwoPartitioning}
\end{lem}
\begin{proof}
Let $v_{C_1}$ and $v_{C_2}$ be the characteristic vectors for the sets $C_1$ and $C_2$, respectively, and $w = (|C_2|/n) v_{C_1} - (|C_1|/n) v_{C_2}$. By the min-max theorem, we have 
\[
\max_{C_1\subseteq V_G,C_2 = V_G\setminus C_1} \!\! \frac{w^\top\Delta_L w}{w^\top w}\leq \|\Delta_L\|.
\] 
Finally, we note that $\Delta_L = L - pJ_n + npI_n$, $w^\top w = \frac{|C_2|^2|C_1|}{n^2} + \frac{|C_2||C_1|^2}{n^2} =\frac{|C_2||C_1|}{n}$ as $|C_1| + |C_2| = n$, $w^\top (-pJ_n + npI_n)w = p|C_1||C_2|$, and $w^\top L w = E_{C_1,C_2}$. 
\end{proof}

When a bound on $\|\Delta_L\|$ is available, Lemma~\ref{lem:TwoPartitioning} can be used to ensure that $G$ is connected. 
In particular, Lemma~\ref{lem:TwoPartitioning} tells us that a graph of size $n$ is connected if $\|\Delta_L\|<np$ for some $0\leq p\leq 1$. 

Since $0\leq E_{C_1,C_2}\leq |C_1||C_2|$, we know that $\left| E_{C_1,C_2} - p |C_1||C_2|\right| \leq \max\{p,1-p\}|C_1||C_2|$.  Therefore, Lemma~\ref{lem:TwoPartitioning} is only a useful bound when $\|\Delta_L\|\ll np$. We can take Lemma~\ref{lem:TwoPartitioning} a step further and bound the number of edges between \emph{any} two sets of vertices of $G$ using $\|\Delta_L\|$. We denote the union of the vertex sets $C_1$ and $C_2$ by $C_1\sqcup C_2$. 

\begin{lem}
Let $G$ be a graph of size $n$ with vertex set $V_G$ and graph Laplacian matrix $L$. For any $0\leq p\leq 1$, we have
$$
\max_{\substack{C_1\subseteq V_G,C_2\subseteq V_G\setminus C_1,\\C_3 = V_G\setminus (C_1\sqcup C_2)}}  
\frac{\left| E_{C_1,C_2} - p |C_1||C_2| \rule{0pt}{2ex} \right|}{|C_1| |C_2| + |C_1| |C_3| + |C_2| |C_3|} \leq \frac{\|\Delta_L\|}{n},
$$
where $\Delta_L$ is defined in~\eqref{eq:DeltaNotation}. 
\label{lem:ThreePartitioning}
\end{lem}
\begin{proof}
For any partitioning of the vertices of $G$ into three sets $C_1$, $C_2$, and $C_3$, we have $2E_{C_1,C_2} = E_{C_1,C_2\sqcup C_3} + E_{C_2,C_1\sqcup C_3} - E_{C_3,C_1\sqcup C_2}$. By Lemma~\ref{lem:TwoPartitioning}, $E_{C_1,C_2\sqcup C_3}$ is bounded between $p(|C_1||C_2|+|C_1||C_3|) \pm \|\Delta_L\||C_1|(|C_2|+|C_3|)/n$, $E_{C_2,C_1\sqcup C_3}$ is bounded between $p(|C_2||C_1|+|C_2||C_3|) \pm \|\Delta_L\||C_2|(|C_1|+|C_3|)/n$, and $E_{C_3,C_1\sqcup C_2}$ is bounded between $p(|C_3||C_1|+|C_3||C_2|) \pm \|\Delta_L\||C_3|(|C_1|+|C_2|)/n$.  Hence, $E_{C_1,C_2}$ deviates from $p|C_1||C_2|$ by less than $\|\Delta_L\|(|C_1| |C_2|+|C_2| |C_3|+|C_1| |C_3|)/(2n)$. 
\end{proof} 

Lemma~\ref{lem:OnePartitioning} and Lemma~\ref{lem:ThreePartitioning} can be combined to obtain the following result. 
\begin{theorem}
\label{thm:relativesize}
Let $G$ be a graph of size $n$ with adjacency matrix $A$ and graph Laplacian matrix $L$. Suppose that $C_1$, $C_2$, and $C_3$ is a partition of the vertices of $G$ into three sets such that (i) $E_{C_1,C_3}\leq \lambda E_{C_3,C_3}$ for some number $\lambda$ and (ii) $|C_2|<|C_1|$.
Then, for any $0\leq p\leq 1$, we have
\[
|C_2|+ |C_3|  \geq \left( \frac{ n  \left( p|C_1| - p\lambda |C_3|  - \lambda \|\Delta_A\| \right)} {\|\Delta_L\| |C_1| }   - 1 \right) |C_3|, 
\]
where $\Delta_A$ and $\Delta_L$ are defined in~\eqref{eq:DeltaNotation}.
\end{theorem} 
\begin{proof}
By Lemma~\ref{lem:OnePartitioning}, $E_{C_3,C_3}$ deviates from $p|C_3|^2$ by less than $\|\Delta_A\| |C_3|$ and, by Lemma~\ref{lem:ThreePartitioning}, $E_{C_1,C_3}$ deviates from $p|C_1||C_3|$ by less than $\|\Delta_L\|(|C_1||C_2|+|C_2||C_3|+|C_3||C_1|)/n$. Since $E_{C_1,C_3}\leq \lambda E_{C_3,C_3}$, we must have 
\[
\begin{aligned} 
p|C_1||C_3| - & \frac{\|\Delta_L\|}{n}(|C_1||C_2|+|C_2||C_3|+|C_3||C_1|) \\ & \qquad \qquad \leq \lambda\left(p|C_3|^2 + \|\Delta_A\||C_3|\right).
\end{aligned} 
\]
By rearranging this inequality, we find that
\[
|C_2|+ |C_3|  \geq \left(\frac{n\left(p|C_1| - p\lambda|C_3| - \lambda\|\Delta_A\|\right)}{\|\Delta_L\||C_1|}-\frac{|C_2|}{|C_1|}\right)|C_3|.  
\]
The result follows as $|C_2|<|C_1|$. 
\end{proof}

\section{Amplification argument}\label{sec:amplification} 
We are finally ready for our amplification argument, which is a way to improve the bounds on $|C_{\pi/2}|$ from~\eqref{eq:FirstBound}. We first write down a new inequality that holds for any stable equilibrium. We write it down using a kernel function $K$ that later allows us to improve our argument with the aid of a computer (see Section~\ref{sec:compOpt}). 
\begin{lem} 
\label{lem:int_inequality}
Let $G$ be a graph with adjacency matrix $A$. Let $K$ be a kernel function defined on $[-\pi,\pi)\times [-\pi,\pi)$, given by 
\[
K(\alpha, \beta) = 
\begin{cases} 
\sin (|\alpha| -|\beta|), & |\alpha|, |\beta| \leq \frac{\pi}{2}, \\ 
-\cos(\alpha), 
     &  |\alpha| \leq \frac{\pi}{2}, |\beta| > \frac{\pi}{2},\\
1, &  |\alpha| > \frac{\pi}{2}.
\end{cases}
\]
Then any stable equilibrium $\theta$  of~\eqref{eq:dynamical} must satisfy $\sum_j A_{jk} K(\theta_j, \theta_k) \geq 0$ for any $1\leq k\leq n$. 
\end{lem} 
\begin{proof} 
Let $k$ be an integer between $1$ and $n$. Due to periodicity, we may assume that the phases, $\theta_j$, take values in the interval $[-\pi,\pi)$. We split the proof into three cases depending on the value of $\theta_k$. 

{\bf Case 1: $0\leq \theta_k\leq \pi/2$.} 
We first show that $\sin(\theta_j-\theta_k)\leq K(\theta_j,\theta_k)$ for all $j$ by checking the three possible subcases: (i) If $0\leq \theta_j\leq \pi/2$ then $\sin(\theta_j-\theta_k) = \sin(|\theta_j|-|\theta_k|) = K(\theta_j,\theta_k)$; (ii) If $|\theta_j|>\pi/2$ then $\sin(\theta_j-\theta_k) \leq 1 = K(\theta_j,\theta_k)$; and (iii) If $-\pi/2 \leq \theta_j<0$ then $\sin(\theta_j-\theta_k) = \sin(|\theta_j|-|\theta_k|-2|\theta_j|)\leq \sin(|\theta_j|-|\theta_k|) = K(\theta_j,\theta_k)$, where the inequality holds because $-\pi/2\leq |\theta_j|-|\theta_k|\leq \pi/2$ and $0\leq 2|\theta_j|\leq \pi$. The inequality follows from the equilibrium condition  $\sum_jA_{jk}\sin(\theta_j-\theta_k) = 0$. 

{\bf Case 2: $ - \pi/2 \leq  \theta_k< 0$.} 
We first show that $\sin(\theta_j-\theta_k)\geq -K(\theta_j,\theta_k)$ for all $j$ by checking the three possible subcases: (i) If $0\leq \theta_j\leq \pi/2$ then $\sin(\theta_j-\theta_k) = \sin(|\theta_j|-|\theta_k|+2|\theta_k|) \geq -\sin(|\theta_j|-|\theta_k|) = -K(\theta_j,\theta_k)$, where the inequality holds because $-\pi/2\leq |\theta_j|-|\theta_k|\leq \pi/2$ and $0\leq 2|\theta_k|\leq \pi$; (ii) If $|\theta_j|>\pi/2$ then $\sin(\theta_j-\theta_k) \geq -1 = -K(\theta_j,\theta_k)$; and (iii) If $-\pi/2 \leq \theta_j<0$ then $\sin(\theta_j-\theta_k) = \sin(-|\theta_j|+|\theta_k|) = -\sin(|\theta_j|-|\theta_k|) = -K(\theta_j,\theta_k)$. The inequality follows from the equilibrium condition  $\sum_jA_{jk}\sin(\theta_j-\theta_k) = 0$. 

{\bf Case 3: $ |\theta_k| > \pi/2$.} 
From the fact that $\theta$ is an equilibrium, we have $\sum_j A_{jk} \sin(\theta_k -\theta_j) = 0$. Moreover,   $\sum_j A_{jk} \cos(\theta_k-\theta_j) \geq 0$, because the diagonal entries of the Hessian matrix must be nonnegative at a {\em stable} equilbrium (see (2.3) of Ling {\em et al.}~\cite{ling2018landscape}). 
By a trigonometric identity, we find that $0 = \sum_k A_{jk} \sin(\theta_k -\theta_j) = \sin(\theta_k)\sum_jA_{jk}\cos(\theta_j) + \cos(\theta_k)\sum_jA_{jk}\sin(\theta_j)$ and hence, 
\begin{equation} 
\sum_jA_{jk}\sin(\theta_j) = \frac{\sin(\theta_k)}{\cos(\theta_k)}\sum_jA_{jk}\cos(\theta_j),
\label{eq:sinExpand} 
\end{equation} 
where we note that $\cos(\theta_k)\neq 0$ as $\theta_k\neq\pm\pi/2$. Moreover, from another trigonometric identity, we have $0\leq\sum_j A_{jk} \cos(\theta_k-\theta_j) = \cos(\theta_k)\sum_jA_{jk}\cos(\theta_j) + \sin(\theta_k)\sum_jA_{jk}\sin(\theta_j)$, which together with~\eqref{eq:sinExpand} gives 
\[
\left(\cos(\theta_k)+ \frac{\sin^2(\theta_k)}{\cos(\theta_k)}\right)\sum_jA_{jk}\cos(\theta_j)\geq 0. 
\]
Multiplying this inequality by $\cos(\theta_k)$ (note that $\cos(\theta_k)< 0$ as $|\theta_k| > \pi/2$) and using $\cos^2(\theta_k) + \sin^2(\theta_k) = 1$, we conclude that $\sum_j A_{jk} \cos(\theta _j)\leq 0$. To reach the desired inequality in the statement of the lemma, we now check the two possible subcases: (i) If $|\theta_j|\leq \pi/2$ then $\cos(\theta_j) = -K(\theta_j,\theta_k)$ and (ii) If $|\theta_j| > \pi/2$ then $\cos(\theta_j) \geq -1 = -K(\theta_j,\theta_k)$. This means that $0\geq\sum_j A_{jk} \cos(\theta _j)\geq -\sum_j A_{jk} K(\theta _j,\theta_k)$ as desired.
\end{proof} 

In preliminary work we have found indications that Lemma~\ref{lem:int_inequality} can be used to prove stronger results than those we report below; see Section~\ref{sec:compOpt} for further discussion. But we are not sure yet how to write down an argument that uses the full strength of Lemma~\ref{lem:int_inequality} in a readily digestible fashion. So for now we use the following simplified lemma instead. It is a key step in our amplification argument. Ultimately it leads to a global synchronization theorem whose proof is comparatively straightforward.   

\begin{lem}
Let $G$ be a graph and $\theta$ a stable equilibrium of~\eqref{eq:dynamical}, and let $0 < \alpha < \beta < \pi/2$. We have 
\[
E_{C_\beta,C_\beta} \geq \sin(\beta-\alpha)E_{C_\beta,\overline{C_\alpha}},
\]
where $C_\alpha$ and $C_\beta$ are defined in~\eqref{eq:ImportantSet}. Here, $\overline{C_\alpha} = V_G \setminus C_\alpha$.
\label{lem:Edges}
\end{lem}
\begin{proof}
This follows from the previous lemma by carefully bounding $K(\theta_j,\theta_k)$ when $k\in C_\beta$ (which implies that $|\theta_k|>\beta$). We check the three possible subcases: (i) If $j\in C_\beta$ then we might have $|\theta_j|>\pi/2$ so the best we can say is $K(\theta_j,\theta_k)\leq 1$; (ii) If $j\in C_\alpha \setminus C_\beta$ (which implies that $|\theta_j|\leq \pi/2$) then either $|\theta_k|>\pi/2$ so that $K(\theta_j,\theta_k) = -\cos(\theta_j)\leq 0$ or $\beta<|\theta_k|\leq \pi/2$ so that $K(\theta_j,\theta_k) = \sin(|\theta_j|-|\theta_k|)\leq 0$ as $|\theta_j|\leq |\theta_k|$. Either way, we have $K(\theta_j,\theta_k)\leq 0$ when $j\in C_\alpha \setminus C_\beta$; and (iii) If $j \not \in C_\alpha$ (which implies that $|\theta_j|< \alpha$) then either $|\theta_k|>\pi/2$ so that $K(\theta_j,\theta_k) = -\cos(\theta_j)\leq -\cos(\alpha) = -\sin(\pi/2-\alpha) \leq -\sin(\beta-\alpha)$ or $\beta<|\theta_k|\leq \pi/2$ so that $K(\theta_j,\theta_k) = \sin(|\theta_j|-|\theta_k|)\leq -\sin(\beta-\alpha)$. Either way, $K(\theta_j,\theta_k)\leq-\sin(\beta-\alpha)$ when $j\not\in C_\alpha$. We conclude that 
\[
K(\theta_j, \theta_k) \leq
\begin{cases}
1, & j \in C_\beta, \\
0, & j \in C_\alpha \setminus C_\beta, \\
-\sin(\beta - \alpha), & j \not \in C_\alpha
\end{cases}
\]
and hence, by Lemma~\ref{lem:int_inequality}, we have
\[
0\leq \!\sum_{k\in C_\beta} \!\sum_j A_{jk} K(\theta_j,\theta_k) \leq  \! \!\underbrace{\sum_{j,k\in C_{\beta}}  \! \!A_{jk}}_{=E_{C_\beta,C_\beta}} - \sin(\beta-\alpha)  \! \!\underbrace{\sum_{j\not\in C_{\alpha},k\in C_\beta}  \! \!A_{jk}}_{=E_{C_\beta,\overline{C_\alpha}}}
\]
as desired.
\end{proof}

We now show that if $\theta$ is a stable equilibrium and $|C_\alpha|$ is small, then $|C_\beta|$ must be even smaller for $0<\alpha<\beta<\pi/2$; otherwise, the oscillators in the set $C_\alpha\setminus C_\beta$ would destabilize the equilibrium (see Figure~\ref{fig:CsetEdges}). Since we know that $|C_\beta|\leq |C_\alpha|$, the next bound is useful when $\|\Delta_L\|/(np) > 1/4$.  
\begin{figure} 
\centering 
\begin{tikzpicture} 
\draw[thick,black] (0,0) circle [radius=2cm];
\filldraw[black] (2,0) circle (3pt);
\filldraw[black] (1.7,1.054) circle (3pt);
\filldraw[black] (1.8,0.87) circle (3pt);
\filldraw[black] (1.6,-1.2) circle (3pt);
\filldraw[black] (1.85,-0.76) circle (3pt);
\filldraw[black] (1.1,-1.67) circle (3pt);
\filldraw[black] (0.5,-1.94) circle (3pt);
\filldraw[black] (2,0) circle (3pt);
\draw[black,dashed] (0,0) -- (.7,0);
\draw[black] (0,0) -- (2.2,1.19);
\draw[black] (0,0) -- (2.2,-1.19);
\draw[black] (0,0) -- (1.5,2);
\draw[black] (0,0) -- (1.5,-2);
\node at (2.25,0) {$1$};
\node at (2.1,0.87) {$2$};
\node at (0.55,-2.2) {$3$};
\node at (1.2,-1.95) {$5$};
\node at (1.9,1.2) {$6$};
\node at (2.1,-0.85) {$7$};
\node at (1.9,-1.3) {$8$};
\node at (.5,.14) {$\alpha$};
\node at (.37,-.2) {$\beta$};
\end{tikzpicture} 
\caption{Here, $C_\alpha = \left\{3,5,6,8\right\}$ and $C_\beta = \left\{3,5\right\}$. We illustrate a hypothetical equilibrium such that $C_{\pi/2}$ is empty and so must be unstable; however, it might be that~\eqref{eq:FirstBound} is not tight enough to show that $|C_{\pi/2}|<1$. To still conclude that $C_{\pi/2}$ is empty, we first show in Lemma~\ref{lem:Edges} that for an equilibrium to be stable there must be enough edges coupling the oscillators in $C_\beta$ together, compared to those between $C_\beta$ and $\overline{C_\alpha}$. For the illustrated equilibrium, we are comparing the number of internal edges between oscillators $3$ and $5$ and the number of outgoing edges to oscillators $1,2$, and $7$.} 
\label{fig:CsetEdges}
\end{figure}
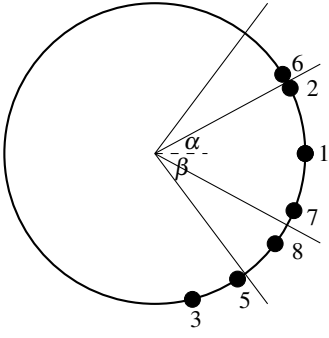 

\begin{cor}
\label{co:improvemnt}
Let $G$ be a graph of size $n$, $\theta$ a stable equilibrium of~\eqref{eq:dynamical}, and $0< p\leq 1$. If for some $0< \alpha < \beta < \pi/2$ we have $|C_\alpha| \leq n/2$, $|C_{\beta} | \leq 2\|\Delta_A\|/p$ and $\sin(\beta-\alpha) \geq 12\|\Delta_A\|/(np)$, then 
$$
|C_{\beta} | \leq \left(\frac{np}{2\|\Delta_L\|} - 1\right)^{-1} |C_\alpha |,
$$
where $\Delta_A$ and $\Delta_L$ are defined in~\eqref{eq:DeltaNotation} and $C_\alpha$ and $C_\beta$ in~\eqref{eq:ImportantSet}.
\end{cor}
\begin{proof} 
Let $\lambda = np/(12\|\Delta_A\|)$. By Lemma~\ref{lem:Edges}, we have $E_{C_\beta,\overline{C_\alpha}}\leq (1/\sin(\beta-\alpha))E_{C_\beta,C_\beta} \leq \lambda E_{C_\beta,C_\beta}$. Hence, by Theorem~\ref{thm:relativesize} (with $C_1 = \overline{C_\alpha}$, $C_2 = C_\alpha\setminus C_\beta$, and $C_3 = C_\beta$), we find that
\[
\begin{aligned} 
|C_\alpha| &\geq \left(\frac{n(p|\overline{C_\alpha}| - p\lambda|C_\beta|-\lambda\|\Delta_A\|)}{\|\Delta_L\| |\overline{C_{\alpha}}|}\right)|C_\beta|\\
& =  \left(\frac{np}{\|\Delta_L\|} -  \frac{n\lambda (p |C_\beta|+\|\Delta_A\|)}{\|\Delta_L\| |\overline{C_{\alpha}}|}\right)|C_\beta|\\
&\geq \left(\frac{np}{2\|\Delta_L\|} - 1\right)|C_\beta|,
\end{aligned} 
\]
where the last inequality holds if $\lambda \leq \frac{\|\Delta_L\||\overline{C_\alpha}|}{n(p|C_\beta| + \|\Delta_A\|)}\left(\frac{np}{2\|\Delta_L\|}+1\right)$. Since $|\overline{C_\alpha}|\geq n/2$ and $|C_{\beta} | \leq 2\|\Delta_A\|/p$, we find that $\lambda = np/(12\|\Delta_A\|)$ satisfies this upper bound. 
\end{proof} 

Note that it is only possible to have an $\alpha$ and $\beta$ such that $\sin(\beta-\alpha)\geq 12\|\Delta_A\|/(np)$ in Corollary~\ref{co:improvemnt} when $\|\Delta_A\|/(np)<1/12$. Corollary~\ref{co:improvemnt} can be used in a recursive fashion to improve the bound on $|C_{\pi/2}|$. Below, we start at $\alpha$ and incrementally increase $\beta$ to conclude that $C_{\pi/2}$ is empty. 

\begin{lemma}
\label{lem:stabilization}
Let $G$ be a graph of size $n$, $\theta$ a stable equilibrium of~\eqref{eq:dynamical}, $0< p\leq 1$, $\|\Delta_A\|/(np)<1/12$, and $\|\Delta_L\|/(np) < 1/4$. If for some $\alpha < \pi/2$ we have (i) $|C_\alpha | < 2\|\Delta_A\|/p$, and (ii)
$$ 
\frac{\pi/2 - \alpha}{\sin^{-1}\!\left(\frac{12\|\Delta_A\|}{np}\right) } > \frac {\log(n/6)}{\log\!\left(\frac{np}{2\|\Delta_L\|} - 1\right)} + 1,
$$
 then $C_{\pi/2}$ is empty and $\theta$ is the all-in-phase state.
\end{lemma}
\begin{proof} 
Set $\beta_k = \alpha + k\sin^{-1}\left(12\|\Delta_A\|/(np)\right)$. Since we need $\beta_k<\pi/2$, we can take $0\leq k\leq M$, where 
\[
M = \Bigg\lceil \frac{\pi/2-\alpha}{\sin^{-1}\left(\frac{12\|\Delta_A\|}{np}\right)}-1\Bigg\rceil.
\]  
By Corollary~\ref{co:improvemnt} and the fact that $|C_\alpha | < 2\|\Delta_A\|/p < n/6$ (which ensures that $|C_\alpha |\leq n/2$), we have 
\[
\begin{aligned}
|C_{\beta_M}| &\leq \left(\frac{np}{2\|\Delta_L\|} - 1\right)^{\!-1}\!\!\!\!|C_{\beta_{M-1}}|\\
&\qquad\qquad \vdots\\ &\leq \left(\frac{np}{2\|\Delta_L\|} - 1\right)^{\!-M}\!\!\!\!|C_{\alpha}|.
\end{aligned} 
\]
Since $|C_{\pi/2}|\leq |C_{\beta_M}|$ and $|C_{\alpha}|\leq n/6$, to conclude that $|C_{\pi/2}|<1$ we need $(np/(2\|\Delta_L\|) - 1)^{-M}n/6<1$, i.e., $M > \log(n/6) / \log(np/(2\|\Delta_L\|) - 1)$ and the result follows. 
\end{proof} 

Finally, we summarize our findings. In particular, we can now provide a list of technical criteria that ensure that the network is globally synchronizing. 
\begin{theorem}
Let $G$ be a graph with $n$ vertices and $0<p<1$. If (i) $\|\Delta_A\|/(np)<1/12$, (ii) $\|\Delta_L\|/(np)<1/4$, and (iii)
$$ 
\frac{\pi/4}{\sin^{-1}\!\left(\frac{12\|\Delta_A\|}{np}\right) } > \frac {\log(n/6)}{\log\!\left(\frac{np}{2\|\Delta_L\|} - 1\right)} + 1,
$$
then the only stable equilibrium of~\eqref{eq:dynamical} is the all-in-phase state. 
\label{thm:MainGeneralResult} 
\end{theorem}
\begin{proof} 
Let $\theta$ be any stable equilibrium of~\eqref{eq:dynamical} on $G$.  By combining~\eqref{eq:LowerBoundRho1} and~\eqref{eq:LowerBoundRho2}, we find that 
\begin{equation} 
\rho_1^6 - (1-2a)\rho_1^4 + 2a^2\rho_1^2 - 2a^4\geq 0, \qquad a = \frac{\|\Delta_A\|}{np}. 
\label{eq:CombinedInequality} 
\end{equation}
Since $a<1/12$ by (i) (which implies that $a<1/5$),~\eqref{eq:CombinedInequality} ensures that $\rho_1^2>a$. Now select $\phi = \pi/4$. By~\eqref{eq:FirstBound}, we find that $|C_{\pi/4}| \leq 2\|\Delta_A\|^2/(np^2\rho_1^2)< 2\|\Delta_A\|/p$ as $\rho_1^2> \|\Delta_A\|/(np)$. By taking $\alpha = \pi/4$ in Lemma~\ref{lem:stabilization} and as (ii) holds, we find that $\theta$ is the all-in-phase state when (iii) is satisfied. 
\end{proof} 

Theorem~\ref{thm:MainGeneralResult} shows that a graph's global synchrony can be ensured by the size of $\|\Delta_A\|$ and $\|\Delta_L\|$ alone. This is particularly beneficial for random networks as $\|\Delta_A\|$ and $\|\Delta_L\|$ are quantities that are studied in the random matrix literature. 

\section{The global synchrony of \Erdos graphs}\label{sec:Erdos}
For an \Erdos random graph with probability $0<p<1$, we have~\cite{ling2018landscape} (also, see Theorem 6.6.1 of Ref.~\cite{Tropp}):
\[
\mathbb{P}\left[ 
\| \Delta_A \| \geq f(n,p)
\right] < 2n^{-1},\quad \mathbb{P}\left[ 
\| \Delta_L \| \geq 2f(n,p)
\right] < 2n^{-1},
\]
where $f(n,p) = 2\sqrt{n \log n \,p(1-p) } + 4(\log n)/3$.
Stronger probability bounds on $\|\Delta_A\|$ are available in the work of F\"{u}redi and Koml\'{o}s~\cite{FK} and Vu~\cite{Vu}; however, the bounds involve implicit constants that are difficult to track down. Since we desire explicit, and not just asymptotic statements, we start by not using these stronger probability bounds. 

Now, let $p > 0.256$ and $n=10^6$. One can verify that $f(n,p)/(np)<1/12$ and $2f(n,p)/(np)<1/4$ so that with probability $>0.999996$, (i) and (ii) in Theorem~\ref{thm:MainGeneralResult} hold. Moreover, one can also check that
\begin{equation} 
\frac{\pi/4}{\sin^{-1}\!\left(\frac{12f(n,p)}{np}\right) } > \frac {\log(n/6)}{\log\!\left(\frac{np}{4f(n,p)} - 1\right)} + 1.
\label{eq:LastIneq} 
\end{equation} 
Therefore, by Theorem~\ref{thm:MainGeneralResult}, an \Erdos graph with $p > 0.256$ and $n=10^6$ is globally synchronizing with probability $>0.999996$. For $n = 10^7$, we find that $p>0.0474$ suffices. 

To ensure that~\eqref{eq:LastIneq} holds as $n\rightarrow\infty$, we see that the left hand side of~\eqref{eq:LastIneq} must grow at least as fast as $\log(n)$. By taking $p = c\log^\gamma(n)/n$ for some $c>0$ and $\gamma>1$, we see that $f(n,p)$ shrinks like $\log^{1/2-\gamma/2}(n)$. Therefore, we need $\gamma>3$, i.e., $$p\gg \log^3(n)/n$$ for this argument to guarantee global synchrony. But as we mentioned, even stronger asymptotic probability bounds are available~\cite{FeigeOfek} on $\|\Delta_A\|$ and $\|\Delta_L\|$, where $f(n,p) = O(\sqrt{np(1-p)})$. With these stronger probability bounds, we find that $$p\gg \log^2(n)/n$$ as $n\rightarrow \infty$ is sufficient to conclude global synchrony of \ER graphs.

\subsection{Optimizing our bounds using a computer}\label{sec:compOpt}
For a given $n$, one can significantly improve the range of $p$ for which the corresponding \ER graph is globally synchronizing by using a computer (see Table~\ref{tab:ValuesOfP}). 
Our computer program can be turned into a proof and thus for $p$ above the thresholds in Table~\ref{tab:ValuesOfP}, the \ER networks are globally synchronizing with probability $>1-4/n$. However, writing the proofs down is unwieldy since the program works with bounds on $|C_\phi|$ for a thousand different values of $\phi$ and iteratively refines those bounds a hundred thousand times over.

By starting with $\rho_1 = 0$ and $\rho_2 = 0$, one can alternate between~\eqref{eq:LowerBoundRho1} and~\eqref{eq:LowerBoundRho2}---in an iterative fashion---to obtain a lower bound on $\rho_1$. The lower bound on $\rho_1$ can be substituted in~\eqref{eq:FirstBound} to give initial upper bounds on $|C_\phi|$ for $0\leq \phi\leq \pi/2$. One can then use Corollary~\ref{co:improvemnt} to progressively improve the bounds on $|C_\phi|$ for $0\leq \phi\leq \pi/2$. To do so, one selects $0<\alpha<\beta<\pi/2$ and attempts to apply  Corollary~\ref{co:improvemnt}. If the application of Corollary~\ref{co:improvemnt} is successful, then one also has $|C_\phi|\leq |C_\beta|$ for $\beta\leq \phi\leq \pi/2$. Since $|C_\phi|$ is integer-valued, any upper bound that is $<1$ implies that $C_\phi$ is empty. We repeat this procedure a hundred thousand times to refine the upper bounds on $|C_\phi|$ for $0\leq \phi\leq \pi/2$. If at any point we have $|C_{\pi/2}|<1$, then we conclude that the \ER graph is globally synchronizing with probability $>1-4/n$. In Table~\ref{tab:ValuesOfP}, we used the explicit value of $f(n,p) = 2\sqrt{n \log n \,p(1-p) } + 4(\log n)/3$ to bound the spectral norms of $\Delta_A$ and $\Delta_L$. 

\begin{table} 
\centering 
\renewcommand{\arraystretch}{1.5}
\begin{tabular}{c||cccc}
$n$ & $10^4$ & $10^5$ & $10^6$ & $10^7$ \\
\hline 
$p$\,\, & \,\,$>0.33237$\,\, & \,\,$>0.07168$\,\, & \,\,$>0.01117$\,\, &\,\,$>0.00157$\\
\end{tabular} 
\caption{The values of $p$ in the \ER random graph model for which we can prove global synchrony for $n = 10^4, 10^5, 10^6$, and $10^7$ with probability $>1-4/n$. We used a computer to recursively apply inequalities in our paper to obtain refined bounds on $|C_{\pi/2}|$. We include this table to demonstrate that our results are meaningful for \ER graphs of finite, practical size. It is possible that these lower bounds on $p$ can be improved by careful optimizations.}
\label{tab:ValuesOfP}
\end{table} 

There are several further improvements to our computer program that we tried: (1) Using stronger probability bounds~\cite{FK,Vu} for $\|\Delta_A\|$; (2) Using Lemma~\ref{lem:int_inequality} instead of Lemma~\ref{lem:Edges}; and (3) Doing additional optimizations to improve the upper bounds for $|C_\phi|$. For example, by selecting triples $0<\alpha<\beta_1<\beta_2<\pi/2$ and proving a generalization of Corollary~\ref{co:improvemnt},  we get bounds for $|C_{\beta_1}|$ in terms of $|C_\alpha|$ and $|C_{\beta_2}|$ and bounds for $|C_{\beta_2}|$ in terms of $|C_\alpha|$ and $|C_{\beta_1}|$. There are similar generalizations for more than three angles. For example, when $n = 10^{20}$, by using Lemma~\ref{lem:Edges} we can only show that $p>1.58\times 10^{-15}$ guarantees that an \ER graph is globally synchronizing with high probability, but with these extra improvements we find that $p>3.50\times 10^{-16}$ suffices. These improvements also provide good evidence---but not a proof---that \ER networks with $p\gg\log(n)/n$ globally synchronize with high probability as $n\rightarrow\infty$. 

\section{Discussion}
We have demonstrated how spectral properties of a graph's adjacency and graph Laplacian matrix can be used to understand the global synchrony of a Kuramoto model with identical oscillators coupled according to a network. For \ER graphs, we prove that $p\gg \log^2(n)/n$ is sufficient to ensure global synchrony with high probability as $n\rightarrow\infty$. As conjectured by Ling, Xu, and Bandeira,~\cite{ling2018landscape} we also believe that the global synchrony threshold is close to the connectivity threshold of $p\sim \log(n)/n$. With the aid of a computer and Lemma~\ref{lem:int_inequality}, we have convincing evidence that \ER networks with 
$p \gg \log(n)/n$ are globally synchronizing with high probability as $n\rightarrow\infty$ and it is a future challenge to write down a formal proof. While Section~\ref{sec:Erdos} focuses on \ER graphs, most of our analysis applies to any network and we hope that it can deliver intriguing results for other random graph models.

\begin{acknowledgments}
Research supported by Simons Foundation Grant 713557 to M.~K., NSF Grants No.~DMS-1513179 and No.~CCF-1522054 to S.~H.~S., and NSF Grants No.~DMS-1818757, No.~DMS-1952757, and No.~DMS-2045646 to A.~T.  We thank Lionel Levine and Mikael de la Salle on MathOverFlow for references on bounding $\|\Delta_A\|$ and $\|\Delta_L\|$ for \ER graphs. 
\end{acknowledgments}

\section*{Data Availability}
The data that supports the findings of this study are available within the article. 


\end{document}